\documentclass[mathscr,12pt]{amsart} 
\renewcommand{\qedsymbol}{\relax}
\input xymatrix.tex 
\input xyarrow.tex  
\usepackage[all]{xy}
\usepackage{verbatim}
\usepackage[mathscr]{eucal}
\usepackage{amsmath,amssymb,amsthm}
\usepackage{amscd}
\usepackage{graphics}

\usepackage[ps,curve,2cell]{xypic}
\UseAllTwocells

\begin{document}


\newtheorem{theorem}{Theorem}[section]
\newtheorem*{TheoremA*}{Theorem A}
\newtheorem*{TheoremB*}{Theorem B}
\newtheorem{proposition}[theorem]{Proposition}
\newtheorem{lemma}[theorem]{Lemma}
\newtheorem{Corollary}[theorem]{Corollary}

\theoremstyle{definition}
\newtheorem{step}{Step}
\newtheorem{Definition}[theorem]{Definition}
\newtheorem{Fact}[theorem]{Fact}
\newtheorem{Conjecture}[theorem]{Conjecture}

\let\tilde=\widetilde
\newcommand{\Z}{\mathbb{Z}}
\newcommand{\noqed}{\def\qedsymbol{}}
\newcommand{\LEFT}{\relax}
\newcommand{\RIGHT}{\relax}
\newcommand{\W}{\mathscr{W}}
\let\iotax=i
\let\xi=\xi
\title{Decompositions involving Anick's spaces}
\author{Brayton Gray}
\address{Department of Mathematics, Statistics and Computer Science\\
University of Illinois at Chicago\\
851 S.~Morgan Street\\
Chicago, IL, 60607-7045}

\email{brayton@uic.edu}



\maketitle

The goal of this work is to continue the investigation
of the Anick fibration and the associated spaces.
Recall that this is a $p$-local fibration sequence:
\[
\Omega^2S^{2n+1}\xrightarrow{\pi_n}S^{2n-1}\xrightarrow{} T\xrightarrow{} \Omega S^{2n+1}
\]
where $\pi_n$ is a compression of the $p^r$th power
map on $\Omega^2S^{2n+1}$. This fibration was first described
for $p\geqslant 5$ as the culmination of a 270 page book \cite{A}.
In \cite{AG}, the authors described an $H$ space structure
for the fibration sequence. Its relationship to
EHP spectra was discussed \cite{G3} as well as first
steps to developing a universal property.

Much work has been done since then to find
a simpler construction, and this was obtained
for $p\geqslant 3$ in \cite{GT}. This new construction also
reproduces the results of \cite{AG}. It is in the
context of these new methods that this work is
developed and we assume a familiarity with
\cite{GT}.

One of the main features of the
construction is a certain fibration sequence:
\[
\Omega G\xrightarrow{h} T\xrightarrow{i} R\xrightarrow{\rho} G
\]
where $h$ has a right homotopy inverse $g\colon T\rightarrow \Omega G$
and the adjoint of~$g$:
\[
\tilde{g}\colon\Sigma T\rightarrow G
\]
also has a right homotopy inverse $f\colon G\rightarrow \Sigma T$.
Together these maps define an $H$ space structure
on $T$ and a co-$H$ space structure on $G$, and both
$G$ and $T$ are atomic. Furthermore $R\in \W_r^{\infty}$, the
class of spaces that are the one point union of
$\bmod\  p^s$ Moore spaces for $r\leqslant s$.

For some applications it would be helpful
to have a better understanding of the map~$\rho$.
In order to accomplish this, we reconstruct a
space $D$ from \cite{A}. $D$ is closely related to $G$.
Although the formal properties of $D$ are not as
simple as $G$ (it is not a co-$H$ space)
other properties are simpler (for example
Theorem~A part \textup{(b)} below). Define
\[
C=\bigvee_{i=1}^{\infty} P^{2np^i+1}\left(p^{r+i-1}\right).
\]
\begin{TheoremA*}
There is a cofibration sequence:
\[
C\xrightarrow{c}G\rightarrow D
\]
and a fibration sequence:
\[
F\rightarrow D\rightarrow S^{2n+1}
\]
such that

\textup{(a)}\hspace*{0.5em}$\displaystyle
H^j(D)=\begin{cases}
		Z/p^r&\mbox{if}\quad j=2n\\
	Z/p&\mbox{if}\ j=2np^s\quad s>0\\
	0&otherwise
\end{cases}$

\textup{(b)}\hspace*{0.5em}$\displaystyle
H^j(F)=
\begin{cases}
	Z_{(p)}&if $j=2ni$\\
	0&otherwise
\end{cases}$

\textup{(c)}\hspace*{0.5em}There is a diagram of fibration sequences:
\[
\begin{CD}
T@=T\\
@V{i}VV
@V{i'}VV\\
R@>>>{W}\\
@V{\rho}VV@VVV\\
G@>>>D
\end{CD}
\]
where $i$ and $i'$ are null homotopic.

\textup{(d)$^{\ast}$}\footnotetext{$^{\ast}$Note that $Z(p)/m$ is isomorphic to $Z/{}_p\nu(m)$ where
$\nu(m)$ is the number of powers of $p$ in
$m$.}\hspace*{0.5em}$\displaystyle
H^j(W)=\begin{cases}
	Z_{(p)}/ip^{r-1}&\mbox{if}\quad j=2ni=2np^s\\
	Z_{(p)}/ip^r&\mbox{if}\quad j=2ni, otherwise\\
	0&otherwise
\end{cases}$

\textup{(e)}\hspace*{0.5em}Furthermore there is a homotopy
commutative diagram of fibration sequences:
\[
\begin{CD}
	S^{2n-1}@>>>T@>>>\Omega S^{2n+1}\\
	@VVV@VVV@VVV\\
	W@=W@>>>PS^{2n+1}\\
	@VVV@VVV@VVV\\
	F@>>>D@>>>S^{2n+1}
\end{CD}
\]
with $\Omega F\simeq S^{2n-1}\times \Omega W$.
\end{TheoremA*}
\begin{TheoremB*}
	If $n>1$, $R\simeq (C\rtimes T)\vee W$ and the
	composition:
	\[
	C\rtimes T\rightarrow R\xrightarrow{\rho} G
	\]
	is homotopic to the composition:
	\[
	C\rtimes T\xrightarrow{1\rtimes g} C\rtimes \Omega
	G\xrightarrow{\omega}C\vee
	G\xrightarrow{c\vee1}G
	\]
	where $\omega$ is the Whitehead product map which
	lies in the fibration sequence:
	\[
	C\rtimes \Omega G\xrightarrow{\omega}C\vee G\xrightarrow{\pi_2}G
	\]
	Furthermore, $\Omega G\simeq \Omega G\times \Omega(C\rtimes\Omega D)$.
\end{TheoremB*}

	In addition, some partial results are
	obtained for $\rho|W$, but much is still unknown.

This paper is organized as follows. In
section~1 we revisit some constructions in~\cite{GT} and sharpen some of
the results. In 
section~2 we embark on a multifaceted
induction, constructing the space $W$ via a
sequence of approximations and prove Theorem~A. Section~3 is devoted
to proving Theorem~B.

\section{}
In the course of the constructions in \cite{GT}, $G$ was
constructed inductively as the union of spaces $G_k$ where
\[
G_k=G_{k-1}\cup CP^{2np^k}\left(p^{r+k}\right)
\]
$G_k$ was constructed as a retract of $\Sigma T^{2np^k}\,$, the
suspension of the $2np^k$ skeleton of~$T$. We need to
make a refinement of this construction. In the proof
of~\textup{4.3(d)} a map
\[
e\colon P^{2np^k}\left(p^{r+k-1}\right)\vee
P^{2np^{k}+1}\left(p^{r+k-1}\right)\rightarrow \Sigma T^{2np^k}
\]
was constructed with the sole property that it induced
an epimorphism in $\bmod\  p$ homology in dimensions
$2np^k$ and $2np^k+1$. The components of $e$ were given as compositions:
\begin{align*}
&	P^{2np^k}\left(p^{r+k-1}\right)\rightarrow
	\Sigma\left(T^{2np^{k-1}-1}\times T^{2np^{k-1}}\times\dots\times
	T^{2np^{k-1}}\right)\xrightarrow{\Sigma\tilde{\mu}'}\Sigma
	T^{2np^k}\\
&	P^{2np^k+1}\left(p^{r+k-1}\right)\rightarrow
	\Sigma\left(T^{2np^{k-1}}\times\dots\times
	T^{2np^{k-1}}\right)\xrightarrow{\Sigma\tilde{\mu}}\Sigma
	T^{2np^k}
\end{align*}
where the middle space in each case is
the suspension of a product of $p$ factors and lies in
$\W_r^{r+k-1}$ and the maps $\tilde{\mu}$ and $\tilde{\mu}'$ are obtained
from the
action given in 4.3(n) for the case $k-1$.
\begin{proposition}\label{prop1.1}
	There is a choice of a map $e$ which
	is a $\bmod\  p$ homology epimorphism in dimensions $2np^k$
	and $2np^k+1$ and such that the diagram:
\[
\begin{CD}
P^{2np^k}\left(p^{r+k-1}\right)\vee
	P^{2np^k+1}\left(p^{r+k-1}\right)@>{e}>> \Sigma T^{2np^k}\\
	@VVV@VVV\\
	\Sigma(T\wedge T)@>H(\mu)>>\Sigma T
\end{CD}
\]
homotopy commutes where $H(\mu)$ is the Hopf construction
on the multiplication $\mu\colon T\times T\rightarrow T$.
\end{proposition}
\begin{proof}
	Since $\tilde{\mu}$ and $\tilde{\mu}'$ are obtained by iteration
	of the restriction of $\mu_{k-1}$ in 4.3(n):
	\[
	T^{2np^{k-1}}
\times
T^{2nmp^{k-1}}\rightarrow T^{2n(m+1)p^{k-1}}
	\]
for $1\leqslant m< p$, it follows that $e$ factors through
\[
\Sigma(\mu)\colon \Sigma\left(T^{2np^{k-1}}\times T^{2n(p-1)p^{k-1}}\right)\rightarrow \Sigma T^{2np^k}.
\]
Now the standard splitting
\begin{multline*}
\Sigma T^{2np^{k-1}}\vee \Sigma T^{2n(p-1)p^{k-1}}\vee
\Sigma\left(T^{2np^{k-1}}\wedge T^{2n(p-1)p^{k-1}}\right)\\
{} \rightarrow
\Sigma\left(T^{2np^{k-1}}\times T^{2n(p-1)p^{k-1}}\right)
\end{multline*}
is induced by the inclusions of the axes and the Hopf
construction on the identity map of the product. Let $e_1$, $e_2$, $e_3$
be the eidempotent self maps of
$\Sigma\left(T^{2np^{k-1}}\times T^{2n(p-1)p^{k-1}}\right)$ corresponding to
these three
retracts. Then in homology we have
\[
1=(e_1)_{\ast}+(e_2)_{\ast}+(e_3)_{\ast}.
\]
However, in dimensions $2np^k$ and $2np^k+1$,
$(e_1)_{\ast}=(e_2)_{\ast}=0$.
Consequently the composition:
\begin{align*}
P^{2np^k}\left(p^{r+k-1}\right)\vee P^{2np^k+1}\left(p^{r+k-1}\right)&
\rightarrow \Sigma\left(T^{2np^{k-1}}\times T^{2n(p-1)p^{k-1}}\right)\\
&\xrightarrow{e_3} \Sigma\left(T^{2np^{k-1}}\times T^{2n(p-1)p^{k-1}}\right)\\
&\rightarrow \Sigma T^{2np^k}
\end{align*}
is also an epimorphism in $\bmod\  p$ homology in
dimensions $2np^k$ and $2np^k+1$. However $e_3$ is the composition:
\begin{align*}
	\Sigma\left(T^{2np^{k-1}}\times T^{2n(p-1)p^{k-1}}\right)&\rightarrow
	\Sigma\left(T^{2np^{k-1}}\wedge T^{2n(p-1)p^{k-1}}\right)\\
	&\xrightarrow{H}\Sigma\left(T^{2np^{k-1}}\times T^{2n(p-1)p^{k-1}}\right)
\end{align*}
where $H$ is the Hopf construction on the identity. Thus
$\Sigma\tilde{\mu}\circ H$ is the Hopf construction on $\tilde{\mu}$:
\begin{align*}
\Sigma\left(T^{2np^{k-1}}\wedge
T^{2n(p-1)p^{k-1}}\right)&\xrightarrow{H}\Sigma\left(T^{2np^{k-1}}
\times T^{2n(p-1)p^{k-1}}\right) \xrightarrow{\Sigma\tilde{\mu}}\Sigma T^{2np^k}
\end{align*}
\end{proof}

We now apply these considerations to the induced
fibration determined by $\tilde{g}$:
\[
\begin{CD}
	T@= T\\
	@VVV@VVV\\
	Q@>>>R\\
	@V{\pi'}VV@VVV\\
	\Sigma T @>{\tilde{g}}>>G
\end{CD}
\]
The structure of the fibration $\pi'\colon Q\rightarrow \Sigma T$ is
completey determined by the action map:
\[
\Omega \Sigma T \times T\rightarrow T
\]
which is given by the composition
\[
\Omega\Sigma T\times T\xrightarrow{\Omega\tilde{g}\times 1}\Omega G\times
T\rightarrow T.
\]
This is the action described in the proof of 4.3(n), so
$Q\simeq \Sigma T\wedge T$ and $\pi'\sim H(\mu)$. We conclude
\begin{proposition}\label{prop1.2}
There is a lifting $\tilde{e}$ of $\tilde{g}e$ to $R$
\[
\xy
\xymatrix
{
P^{2np^k}\left(p^{r+k-1}\right)\vee P^{2np^{k}+1}\left(p^{r+k-1}\right)
\drto^{e}
\rto^{}
&\Sigma T\wedge T
\dto_{H(\mu)}
\rto^{}&R
\dto^{}\\
&\Sigma T\rto^{\tilde{g}}&G
}
\endxy
\]

	We will designate the components of $\tilde{g}e$ as
	\begin{align*}
		&a_k\colon P^{2np^k}\left(p^{n+k-1}\right)\rightarrow G_k\\
		&c_k\colon P^{2np^k+1}\left(p^{r+k-1}\right)\rightarrow G_k
	\end{align*}
	and their lifts to $R_k$ as $\widetilde{a}_k$ and $\widetilde{c}_k$
	respectively.
\end{proposition}
\begin{proposition}\label{prop1.3}
	There is a homotopy commutative
	diagram of cofibrations sequences:
	\[
	\begin{CD}
		P^{2np^k}\!\!\left(p^{r+k}\right)\!
		@>{p^{r+k-1}}>>
		P^{2np^k}
		\!\!\left(p^{r+k}\right)\!
		@>{d}>>
P^{2np^k}\!\!\left(p^{r+k-1}\right)\!
\!{}\vee{}\! P^{2np^k+1}\!\!\left(p^{r+k-1}\right)\!
\\
		@|
		@VV{\beta_k}V
		@VV{a_k\vee c_k}V
		\\
		P^{2np^k}\!\!\left(p^{r+k}\right)\!
		@>{\alpha_k}>>
		G_{k-1}
		@>>>
		G_k
	\end{CD}
	\]
\end{proposition}
\begin{proof}
	The homotopy fiber of the projection $G_k\rightarrow P^{2np^{k+1}}(p^{r+k})$
	is the relative James construction $(G_{k-1}, P^{2np^k}(p^{r+k}))_{\infty}$ 
	(see~\cite{G1}) which is $G_{k-1}+{}$cells of dimension${}\geqslant 2np^k+2n-1$.
	Thus the map $\beta_k$ exists. The next term
	is the James construction $(P^{2np^k+1}(p^{r+k}))_{\infty}$. The
	upper cofibration sequence is completely determined by
	the fact that the composition
	\[
	\begin{CD}
		P^{2np^k}\left(p^{r+k-1}\right)\vee P^{2np^k+1}\left(p^{r+k-1}\right)\xrightarrow{a_k\vee c_k} G_k\rightarrow P^{2np^{k+1}}\left(p^{r+k}\right)
	\end{CD}
	\]
	is an epimorphism in homology.
\end{proof}
\begin{proposition}\label{prop1.4}
	There is a unique lifting $\widetilde{\beta}_k$ of $\beta_k$ to
	$R_{k-1}$:
\[
\xy
\xymatrix
{
&R_{k-1}\ar[dd]\\
P^{2np^k}\left(p^{r+k}\right)
\ar[ur]^{{\widetilde{\beta}_k}}
\ar[dr]^{{\beta_k}}&\\
&G_{k-1}
}
\endxy
\]
\end{proposition}
\begin{proof}
The existance follows since $R_{k-1}$ is a
pullback:
\[
\begin{CD}
	R_{k-1}@>>> R_k\\
	@VVV@VVV\\
	G_{k-1}@>>> G_k
\end{CD}
\]
and $(\widetilde{a}_k\vee \widetilde{c}_k)d\colon P^{2np^k}(p^{r+k})\rightarrow R_k$ is a lifting
of the composition
\[
P^{2np^k}\left(p^{r+k}\right)\xrightarrow{\beta_k}G_{k-1}\rightarrow G_k
\]
by 1.3. To prove uniqueness, suppose we have two
liftings $\widetilde{\beta}_k$ and ${\widetilde{\beta}_k}'$. Their difference consequently
factors through $T$. But any map
\[
P^{2np^k}\left(p^{r+k}\right)\rightarrow T
\]
is necessarily trivial in $\bmod\  p$ cohmoology, for
$H^{2np^k-1}(T; \Z/p)$ is\linebreak[4]
 decomposable and the $\bmod\  p^{r+k}$
Bockstein is nontrivial in\linebreak[4]
 $H^{2np^k-1}(T; \Z/p)$.
It follows that the difference $\widetilde{\beta}_k-\widetilde{\beta}_k'$ factors
through $T^{2np^k-2}$ and hence though $\Omega G_{k-1}$ by~\cite[4.3(b)]{GT}. Consequently the difference is trivial in $R_{k-1}$.
\end{proof}

The following result is a special case of~\cite[2.3]{GT}
\begin{theorem}\label{theor1.5}
Suppose all spaces are localized at a
prime $p>2$, and in the diagram:
\[
\xy
\xymatrix
{
&F\dto{}\ar@{=}[r]&F\dto{}\\
S^{m-1}\rto{}&E_0\rto{}\dto{}&E\dto{}\\
S^{m-1}\uto^{p}\ar@{-}[r]&B\rto{}&B\cup e^m
}
\endxy
\]
the middle column is a pullback and the bottom row is a cofibration. Then
\[
\left[E,BW_n\right]\rightarrow \left[E_0,BW_n\right]
\]
is onto.
\end{theorem}
\section{}
In this section we will construct the fibration
\[
T\rightarrow W_k\rightarrow D_k
\]
and prove Theorem~A. The spaces $D_k$ were first
considered in	~\cite{A} and their relationship to $G_k$
was discussed in~\cite{AG}. We will construct them
directly from the ideas of~\cite{GT}.
We begin with
\[
C_k=\bigvee_{i=1}^k P^{2np^i+1}\left(p^{r+i-1}\right)
\]
and define $c\colon C_k\rightarrow G_k$ by $c\mid P^{2np^i+1}(p^{r+i-1})=c_i$.
Now define $D_k$ as the cofiber:
\[
C_k\xrightarrow{c}G_k\rightarrow D_k.
\]
Since $c_i\colon P^{2np^i+1}(p^{r+i-1})\rightarrow G_i\rightarrow G_k$ is an integral
homology mono\-morphism, we immediately have
\begin{proposition}\label{prop2.1}
	$H_i(D_k)=\begin{cases}
		Z/p^r&\mbox{if}\quad i=2n\\
		Z/p&\mbox{if}\quad i=2np^j\quad 1\leqslant j\leqslant k\\
		0&\mbox{otherwise}.
	\end{cases}$
\end{proposition}
	By construction, the map $c$ lifts to a map
	\[
	C_k\xrightarrow{\widetilde{c}}R_k\rightarrow E_k
	\]
	where $E_k$ is described in \cite[4.3(h)]{GT} and we have
\[
\xy
\xymatrix
{
&E_k\rto^{} 
\dto{}&J_k\rto^{}
\dto_{\xi_k}&\dto{}F_k\\
(\ast)\qquad C_k
\urto^{\widetilde{c}}
\rto^{c}&G_k\dto{}\rto&D_k\dto{}\ar@{=}[r]&D_k\dto\\
&S^{2n+1}\left\{p^r\right\}\ar@{=}[r]&S^{2n+1}\left\{p^r\right\}\rto&S^{2n+1} 
}
\endxy
\]
where the diagram of vertical fibrations defines the
spaces $J_k$ and $F_k$. We are about to embark on a
multipart induction and wish to make one observation
first. Consider the Serre spectral sequence for the
homology of the fibration:
\[
\Omega S^{2n+1}\rightarrow F_k\rightarrow D_k
\]
where
\[
E^2_{p,q}=H_p\left(D_k;H_q\left(\Omega S^{2n+1}\right)\right)
\]
This is only nonzero when both $p$ and $q$ are
divisible by $2n$. Hence $E^2_{p,q}\cong E^{\infty}_{p,q}$ and $H_i(F_k)=0$
unless $i$ is divisible by $2n$. In particular
\begin{equation}
\tag{$\ast\ast$}
H_{2np^k}(F_k)\rightarrow H_{2np^k}(D_k)\ \mbox{is an epimorphism}.
\end{equation}
\begin{theorem}\label{2.2}
	Let $k\geqslant 0$. Then

	\textup{(a)}\hspace*{0.5em}$H_r(F_k)=\begin{cases}
		Z_{(p)}&\mbox{if}\quad $r=2ni$\\
		0&\mbox{otherwise}
	\end{cases}$\\
	and the homomorphism $H_{2ni}(F_{k-1})\rightarrow H_{2ni}(F_k)$
	has degree $p$ if $i\geqslant p^k$ and degree $1$ if $i<p^k$.

	\textup{(b)}\hspace*{0.5em}There is a map $\theta_i\colon
	P^{2ni}(p^r)\rightarrow J_k$ for each $i\geqslant p^k$
	such that the composition:
	\[
	P^{2ni}\left(p^r\right)\xrightarrow{\theta_i} J_k\rightarrow F_k
	\]
	induces an isomorphism in $H^{2ni}(\mbox{\ }; Z/p)$.

	\textup{(c)}\hspace*{0.5em}There is a map
	$\widetilde{\gamma}_k\colon J_k\rightarrow BW_n$ such that the
	composition:
	\[
	\Omega^2 S^{2n+1}\rightarrow
	J_k\xrightarrow{\widetilde{\gamma}_k}BW_n
	\]
	is homotopic to the map $\nu\colon \Omega^2 S^{2n+1}\rightarrow
	BW_n$ $($see~\cite{G2}$)$.

	\textup{(d)}\hspace*{0.5em}Let $W_k$ be the homotopy fiber of
	$\widetilde{\gamma}_k$. Then we have
	a homotopy commutative diagram of vertical fibration
	sequences:
	\[
	\begin{CD}
		T@=T@>>>\Omega S^{2n+1}\\
		@VVV @VVV @VVV\\
		R_k@>>>W_k@>>>F_k\\
		@VVV @VVV @VVV\\
		G_k@>>>D_k@= D_k 
	\end{CD}
	\]
	and two diagrams of fibration sequences:
	{\Small\[
	\begin{matrix}
		\begin{CD}
			S^{2n-1}@>>> \Omega^2S^{2n+1}@>{\nu}>> BW_n\\
			@VVV @VVV @|\\
			W_k@>>>J_k@>{\widetilde{\gamma}_k}>> BW_n\\
			@VVV @VVV\\
			F_k@= F_k
		\end{CD}&&&
		\begin{CD}
			S^{2n-1}@>>>T@>>>\Omega S^{2n+1}\\
			@VVV @VVV\\
			W_k@=W_k\\
			@VVV @VVV\\
			F_k@>>>D_k@>>>S^{2n+1}
		\end{CD}
	\end{matrix}
	\]}

	\textup{(e)}\hspace*{0.5em}$\Omega F_k\simeq S^{2n-1}\times \Omega
	W_k$

	\textup{(f)}\hspace*{0.5em}the homomorphism
	$H^r(F_k)\rightarrow H^r(W_k)$ is an epimorphism and
	\[
	H^m(W_k)=\begin{cases}
		Z_{(p)}/ip^{r-1}&\mbox{if}\quad m=2ni=2np^s\quad 0<s\leqslant
		k\\
		Z_{(p)}/ip^r&\mbox{if}\quad m=2ni > 2n\ \mbox{otherwise}\\
		0&\mbox{otherwise}
	\end{cases}
	\]

	\textup{(g)}\hspace*{0.5em}The image of the homorphism:
	\[
	H^{2np^{k+1}}(W_k)\rightarrow H^{2np^{k+1}}(T)
	\]
	has order $p$.

	\textup{(h)}\hspace*{0.5em}The map $\alpha_{k+1}$ lifts to a map
	$\widetilde{\alpha}_{k+1}\colon P^{2np^{k+1}}(p^{r+k+1})\rightarrow
	R_k$
	such that the composition:
	\[
	P^{2np^{k+1}}\left(p^{r+k+1}\right)\xrightarrow{\widetilde{\alpha}_{k+1}}
	R_k\rightarrow W_k
	\]
	is nonzero in integral cohomology.
\end{theorem}
\begin{proof}
	We prove these results inductively on $k$ using
	earlier results for a given value of $k$ and all
	results for lower values of~$k$. In case $k=0$, \textup{(a)} is 
	well known (see~\cite{CMN}).
\end{proof}
\begin{proof}[Proof of \textup{(b)}]
	In case $k=0$, this is~\cite[3.1]{GT}. Suppose
	$k>0$. Since the homomorphism
	\[
	H_{2np^k}(F_k)\rightarrow H_{2np^k}(D_k)
	\]
	is onto by $(\ast\ast)$ and $H_{2np^k}(F_k)$ is free on one
	generator, the homorphism:
\[
H_{2np^k}\left(F_k;\Z/p\right)\rightarrow H_{2np^k}\left(D_k; \Z/p\right)
\]
is an isomorphism. Now consider the diagram:
\[
\xy
\xymatrix
{
&E_k\dto{}\rto{}&J_k\dto{}\rto{}&F_k\dto{}\\
P^{2np^k}\left(p^{r+k-1}\right)
\urto^{\widetilde{a}_k}\rto^{\makebox[20pt]{}a_k}&G_k\rto{}&
D_k\ar@{=}[r]
&D_k
}
\endxy
\]

Since the lower composition also induces an
isomorphism in\linebreak[4]
 $H_{2np^k}(\mbox{\ }; \Z/p)$, we conclude that
the upper composition does as well. Let $\theta_{p^k}$ be
the composition:
\[
P^{2np^k}\left(p^r\right)\xrightarrow{}P^{2np^k}\left(p^{r+k-1}\right)\xrightarrow{\widetilde{a}_k}E_k\xrightarrow{}J_k.
\]
This satisfies \textup{(b)} in case $i=p^k$. We now construct $\theta_{m+1}$
for $m\geqslant p^k$ by induction. Having constructed $\theta_m$,
consider the diagram of vertical fibration
sequences:
\[
\begin{CD}
	P^{2(m+1)n}\!\left(p^r\right)\!\!\!@>>>\!\!\!P^{2mn}\!\left(p^r\right)\!{}\rtimes{}\!\Omega P^{2n+1}\!\left(p^r\right)\!\!\!@>>>\!\!\!J_k\!\!\!@>>>\!\!\!F_k\\
@.	@VVV @V{\xi_k}VV @VVV\\
@.	P^{2n+1}\!\left(p^r\right) \vee P^{2mn}\!\left(p^r\right)\!\!\!@>{\iota \vee\xi_k\theta_m}>>\!\!\!D_k\!\!\!@=\!\!\!D_k\\
@.	@VVV @VVV @VVV\\
@.	P^{2n+1}\!\left(p^r\right)\!\!\!@>>>\!\!\!S^{2n+1}\!\left\{p^r\right\}\!\!\!@>>>\!\!\!S^{2n+1}
\end{CD}
\]
\noindent
We will compare the $\bmod\  p$ cohomology spectral sequences
for the first and last fibration, and in particular,
the differential:
\[
\begin{CD}
	d_{2n+1}\colon E^{2n+1}_{0,2(m+1)n}@>>> E^{2n+1}_{2n+1, 2mn}.
\end{CD}
\]
In the righthand spectral sequence, the differential
is an isomorphism as both groups are $Z/p$ and the
dimension of~$D_k$ is less than\linebreak[4]
 $2(m+\nobreak1)n$. The map
of fibrations induces the following homomorphism on
$E^{2n+1}_{2n+1,2mn}$ (where the coefficients are $Z/p$):
\begin{multline*}
H^{2mn}\left(F_k\right)\otimes H^{2n+1}\left(S^{2n+1}\right)\rightarrow\\
 H^{2mn}\left(P^{2mn}\left(p^r\right)\rtimes \Omega P^{2n+1}\left(p^r\right)\right)\otimes H^{2n+1}\left(P^{2n+1}\left(p^r\right)\right).
\end{multline*}
Since the composition $\pi_k\theta_m$ induces an isomorphism
in $H^{2mn}(\mbox{\ }; Z/p)$, this homomorphism is an isomorphism as well. It follows that the homomorphism induced on
$E^{2n+1}_{0,2(m+1)n}$ is nonzero and hence an isomorphism. Thus
$\pi_k\theta_{m+1}$ induces an isomorphism in $H^{2(m+1)n}(\mbox{\ }; Z/p)$.
\end{proof}
\begin{proof}[Proof of \textup{(c)}]
	In case $k=0$, this is \cite[3.5]{GT}. Suppose
	that $k>0$. Write $F_k(m)$ for the 2mn skeleton of~$F_k$ and
	$J_k(m)$ for the total space of the induced fibration
	over $F_k(m)$:
\[
\begin{CD}
	\Omega^2S^{2n+1} @= \Omega^2S^{2n+1}\\
	@V{i_k}VV @VVV\\
	J_k(m)@>>> J_k\\
	@VVV @VVV\\
	F_k(m)@>>> F_k
\end{CD}
\]
We will construct a compatible sequence of maps:
\[
\widetilde{\gamma}_k(m)\colon J_k(m)\rightarrow BW_n
\]
with $\widetilde{\gamma}_k(m)i_k\sim \nu\colon \Omega^2S^{2n+1}\rightarrow BW_n$ for a fixed $k$ and $m\geqslant 1$.
Since $D_k=D_{k-1}\cup CP^{2np^k}$, the pair $(F_k, F_{k-1})$ is $2np^k-1$
connected. Consequently if $m<p^k$, $F_{k-1}(m)=F_k(m)$
and $J_{k-1}(m)=J_k(m)$. We begin the induction on $m$ by
defining $\widetilde{\gamma}_k(m)=\widetilde{\gamma}_{k-1}(m)$ when $m<p^k$. Now
$F_k(m)=F_k(m-1)\cup_{\gamma_m} e^{2mn}$. We wish to apply Theorem~\ref{theor1.5} 
to the diagram:
\[
\begin{CD}
	\Omega^2S^{2n+1}@= \Omega^2S^{2n+1}\\
	@VVV @VVV\\
	J_k(m-1)@>>>J_k(m)\\
	@VVV @VVV\\
	F_k(m-1)@>>>F_k(m)
\end{CD}
\]
It suffices to show that there
is a lifting $\gamma'_m$ of $\gamma_m$ which is divisible by~$p$.
\[
\xy
\xymatrix
{
&J_k(m-1)\dto{}\\
S^{2mn-1}
\urto^{\gamma'_m}
\rto^{\makebox[50pt]{\kern-10pt\scriptsize\hfill$\gamma_m$\hfill}}
&F_k(m-1)
}
\endxy
\]
In fact, we will construct a lifting $\gamma'_m$ of $\gamma_m$ which is
divisible by~$p^r$. 

The composition:
\[
S^{2mn-1}\rightarrow P^{2mn}\left(p^r\right)\xrightarrow{\theta_m} J_k(m)\xrightarrow{\pi_k} F_k(m)
\]
factors through $F_k(m-1)$ for dimensional
reasons:
\[
\begin{CD}
	S^{2mn-1}@>>> P^{2mn}\left(p^r\right)\\
	@V{x'}VV @VV{\pi_k\theta_m}V\\
	F_k(m-1)@>>>F_k(m)
\end{CD}
\]
with $p^rx'\sim\gamma_m$, since $\pi_k\theta_m$ induces an isomorphism
in $H^{2mn}(\mbox{\ }; \Z/\rho)$. (For complete details, apply \cite[3.2]{GT}
with $M\!{}={}\!S^{2mn-2}\!$, $X\!{}={}\!F_k(m-\nobreak1)$, $f=\gamma_m$, $x=\pi_k\theta_m$, and $s=r$).
Since $J_k(m-1)$ is a pullback, $x'$ factors through
$J_{k}(m-1)$:
\[
\begin{CD}
	S^{2mn-1}@>>>P^{2mn}\left(p^r\right)\\
	@V{\gamma'_m}VV @V{\theta_m}VV\\
	J_k{(m-1)}@>>> J_k(m)\\
	@V{\pi_{k}}VV @V{\pi_k}VV\\
	F_{k}(m-1)@>>> F_k(m)
\end{CD}
\]
and $\pi_{k-1}\gamma'_m=x'$ so $p^r\pi_{k-1}\gamma'_m\sim \gamma_m$.
Thus we have constructed
$\widetilde{\gamma}_k\colon J_k\rightarrow BW_n$ for each $k\geqslant 1$. $\tilde{\gamma}_k|_{J_{k-1}}$ may not be
homotopic to $\widetilde{\gamma}_{k-1}$, but they are homotopic
on $J_{k-1}(m)$ for $m<p^k$. Thus we may define
\[
\widetilde{\gamma}_{\infty}\colon J_{\infty}\rightarrow BW_n
\]
by taking the direct limit of the $\widetilde{\gamma}_k$ and then
redefine $\widetilde{\gamma}_k$ as the restriction of~$\tilde{\gamma}_{\infty}$.
\end{proof}
\begin{proof}[Proof of \textup{(d)}]
	The map $\nu_k\colon E_k\rightarrow BW_n$ defined
	in \cite[4.3(h)]{GT} was an arbitrary map such that
	the composition:
	\[
	\Omega^2S^{2n+1}\rightarrow\Omega S^{2n+1}\left\{p^r\right\}\rightarrow E_0\rightarrow E_k\xrightarrow{\nu_k} BW_n
	\]
	is homotopic to $\nu$. Since $J_0=E_0$, $\widetilde{\gamma}_0=\nu_0$
	and $\widetilde{\gamma}_k$ is an arbitrary extension of $\widetilde{\gamma}_{k-1}$, we can 
redefine $\nu_k$ as the composition
\[
E_k\rightarrow J_k\xrightarrow{\widetilde{\gamma}_k}BW_n
\]
from which it follows that we have a
commutative diagram of fibration sequences
\[
\begin{CD}
	R_k@>>>E_k@>{\nu_k}>> BW_n\\
	@VVV @VVV @|\\
	W_k@>>>J_k@>{\widetilde{\gamma}_k}>> BW_n
\end{CD}
\]
where $W_k$ is the fiber of $\widetilde{\gamma}_k$.
Consequently the square:
\[
\begin{CD}
	R_k@>>> W_k\\
	@VVV @VVV\\
	G_k@>>> D_k
\end{CD}
\]
is the composition of two pullback squares:
\[
\begin{CD}
	R_k@>>> E_k@>>> G_k\\
	@VVV @VVV @VVV\\
	W_k @>>> J_k @>>> D_k
\end{CD}
\]
so it is a pullback square and first diagram in~\textup{(d)}
is a diagram of vertical fibration sequences. The
second diagram follows from the definition of $W_k$
and the third is a combination of the first two.
\end{proof}
\begin{proof}[Proof of \textup{(e)}]
Extending the third diagram of~\textup{(d)} to the left
yields a diagram:
\[
\begin{CD}
	\Omega^2S^{2n+1} @>{\pi_n}>> S^{2n-1}\\
	@VVV @VVV\\
	\ast @>>> W_k\\
	@VVV @VVV\\
	\Omega^2S^{2n+1}@>>>F_k
\end{CD}
\]
Both horizontal maps have degree $p^r$ in the lowest
dimension, so $W_k$ is $4n-2$ connected and the map
$S^{2n-1}\rightarrow W_k$ is null homotopic. From this it follows
that
\[
\Omega F_k\simeq S^{2n-1}\times \Omega W_k
\]
\end{proof}
\begin{proof}[Proof of \textup{(f)}]
	Let $\phi\colon\Omega S^{2n+1}\rightarrow F_k$ be the connecting
	map in the fibration that defines $F_k$. Let
	$u_i\in H^{2ni}(\Omega S^{2n+1})$ be the generator dual to the $i^{\text{th}}$
	power of a generator in $H_{2n}(\Omega S^{2n+1})$. Then
	\[
	u_i u_j=\left(\begin{matrix}
		i+j\\
		i
	\end{matrix}\right)u_{i+j}.
	\]
	Choose generators $e_i\in H^{2ni}(F_k)$ so that
	\[
	\phi^{\ast}\left(e_i\right)=
	\begin{cases}
		p^{r+d}u_i&\mbox{if}\quad p^d\leqslant i<p^{d+1},\ d\leqslant
		k\\
		p^{r+k}u_i&\mbox{if}\quad i\geqslant p^k.
	\end{cases}
	\]
	Since $\phi^{\ast}$ is a monomorphism, it is easy to check
	that
	\[
	e_1e_{i-1}=
	\begin{cases}
		ip^{r-1}e_i&\mbox{if}\quad i=p^s\quad 0<s\leqslant k\\
		ip^r e_i&\mbox{otherwise}.
	\end{cases}
	\]
	It now follows from the integral cohomology
	spectral sequence for the fibration:
	\[
	S^{2n-1}\rightarrow W_k\rightarrow F_k
	\]
	that
	\[
	d_{2n}\left(e_{i-1}\otimes u\right)=
	\begin{cases}
		ip^{r-1}e_i&\mbox{if}\quad i=p^s\quad 0<s\leqslant k\\
		ip^r e_i&\mbox{otherwise}.
	\end{cases}
	\]
	From this one can read off the cohomology groups
	of~$W_k$ since $H^i(F_k)=Z_{(p)}$ or $0$ according as to
	whether $j$ is a multiple of~$2n$.
\end{proof}
\begin{proof}[Proof of \textup{(g)}]
	From \textup{(d)} we have a homotopy commutative
	square:
	\[
	\begin{CD}
		T@>>> \Omega S^{2n+1}\\
		@VVV @VV{\phi}V\\
		W_k@>>> F_k
	\end{CD}
	\]
	Applying cohomology we get:
	\[
	\xy
	\xymatrix
	{
	H^{2np^{k+1}}(T)&\lto{} H^{2np^{k+1}}(\Omega S^{2n+1})\\
	H^{2np^{k+1}}(W_k)\uto{}&\lto{} H^{2np^{k+1}}(F_k)\uto{}\\
	}
	\endxy
	\]
	which evaluates as:
	\[
	\xy
	\xymatrix
	{
	Z\big/p^{r+k+1}&\lto{} Z_{(p)}\\
	Z\big/p^{r+k+1}\uto^{}&\lto{}Z_{(p)}\uto_{p^{r+k}}\\
	}
	\endxy
	\]
where the two horizontal arrows are epimorphisms.
	It follows that the homorphism $H^{2np^{k+1}}(W_k)\rightarrow
	H^{2np^{k+1}}(T)$
	has image of order~$p$.
\end{proof}
\begin{proof}[Proof of \textup{(h)}]
	Recall from the proof of \cite[4.3(c)]{GT} that
	the composition:
	\[
	P^{2np^{k+1}}\left(p^{r+k+1}\right)\rightarrow
	T\big/T^{2np^{k+1}-2}\rightarrow R_k\rightarrow G_k
\]
is homotopic to $\alpha_{k+1}$. Write $\widetilde{\alpha}_{k+1}$ for the
composition
of the first two maps, so we get a homotopy
commutative diagram:
\[
\begin{CD}
T^{2np^{k+1}}@>>>
P^{2np^{k+1}}\left(p^{r+k+1}\right)@>{\widetilde{\alpha}_{k+1}}>>R_k\\
@VVV @VVV @VVV\\
T^{2np^{k+1}}\big/T^{2np^{k+1}-2}@>>>
		T\big/T^{2np^{k+1}-2}@>>>W_k
	\end{CD}.
	\]
	By \textup{(g)}, the composition on the left and bottom is
	nonzero in integral cohomology, so the
	composition on the top and right is nonzero
	in integral cohomology also.
\end{proof}
\begin{proof}[Proof of \textup{(a)} in the case $k+1$]
	We consider the diagram
	\[
	\xy
	\xymatrix
	{
&&&&\Omega S^{2n+1}\dto^{\delta_{k+1}}\\
	P^{2np^{k+1}}\left(p^{r+k+1}\right)
\ar@<1ex>[urrrr]^{\Gamma}
\drto^{\alpha_{k+1}}
\ar@<0ex>[r]^{\kern34pt\widetilde{\alpha}_{k+1}}
&R_k\dto{}\rto{}&W_k\dto{}\rto{}&F_k\dto{}\rto{}&F_{k+1}\dto{}\\
	&G_k\ar@{-}[r]
&D_k\ar@{=}[r]
&D_k\ar@{-}[r]
&D_{k+1} 
	}
	\endxy
	\]
	where the map $\Gamma$ exists since the lower composite
	factors as
	\[
	P^{2np^{k+1}}\left(p^{r+k+1}\right)\xrightarrow{\alpha_{k+1}}
	G_k\rightarrow G_{k+1}\rightarrow D_{k+1}
	\]
We now show that the homorphism:
\[ 
H^{2np^{k+1}}\left(F_{k+1}\right)\rightarrow H^{2np^{k+1}}\left(F_k\right)
\]
is not an epimorphism. If it were, the entire
composition:
\[
H^{2np^{k+1}}\left(F_{k+1}\right)\rightarrow
H^{2np^{k+1}}\left(P^{2np^{k+1}}\left(p^{r+k+1}\right)\right)
\]
would be nonzero by \textup{(f)} and \textup{(h)}. But $\delta_{k+1}$
factors as
\[
\Omega S^{2n+1}\xrightarrow{\delta_k} F_k\rightarrow F_{k+1}
\]
so the image of $\delta_{k+1}^{\ast}\colon H^{2np^{k+1}}(F_{k+1})\rightarrow
H^{2np^{k+1}}(\Omega S^{2n+1})$
is divisible by $p^{r+k}$ by~\textup{(a)}. This implies that
\[
\Gamma^{\ast}\colon H^{2np^{k+1}}\left(\Omega S^{2n+1}\right)\rightarrow
H^{2np^{k+1}}\left(P^{2np^{k+1}}\left(p^{r+k+1}\right)\right)
\]
is an epimorphism. This contradicts Hopf  invariant one
as follows. The composition:
\[
P^{2np^{k+1}}\left(p^{r+k+1}\right)\rightarrow \Omega
S^{2n+1}\xrightarrow{H_{p^k}} \Omega S^{2np^k+1}
\]
would also be nonzero in cohomology so the Whitehead
element
\[
\omega_{np^k}\in \pi_{2np^{k+1}-3}\left(S^{2np^k-1}\right)
\]
would be divisible by $p^{r+k+1}$. We have thus shown that
\[
H^{2np^{k+1}}\left(F_{k+1}\right)\rightarrow H^{2np^{k+1}}\left(F_k\right)
\]
is not onto. From the Serre spectral sequence for
the cohomology of the fibration
\[
\Omega S^{2n+1}\rightarrow F_{k+1}\rightarrow D_{k+1}
\]
we obtain the following exact 
sequence:
\begin{align*}
0\leftarrow H^{2np^{k+1}+1}\left(F_{k+1}\right)
\leftarrow Z/p
\leftarrow H^{2np^k+1}\left(F_k\right)\leftarrow
H^{2np^{k+1}}\left(F_{k+1}\right)\leftarrow 0
\end{align*}
where $H^{\ast}(F_k)$ is obtained from the restriction of the
fibration to~$D_k$. It follows that
\begin{align*}
	H^{2np^{k+1}+1}\left(F_{k+1}\right)&=0\\
	H^{2np^{k+1}}\left(F_{k+1}\right)&\simeq Z_{(p)}\\
	\text{and}\quad H^{2np^{k+1}}\left(F_{k+1}\right)&\rightarrow
	H^{2np^{k+1}}\left(F_k\right)
\end{align*}
has degree $p$, and $(\delta_{k+1})^{\ast}$
has degree $p^{r+k+1}$ in dimension $2np^{k+1}$. We now switch to integral
homology and use the principal action
\[
\begin{CD}
	\Omega S^{2n+1}\times \Omega S^{2n+1}@>>> \Omega S^{2n+1}\\
	@VVV @VVV\\
	\Omega S^{2n+1}\times F_{k+1}@>>> F_{k+1}
\end{CD}
\]
to study $H_{2ni}(F_{k+1})$ when $i>p^{k+1}$. Observe that
the new generator~$e_i$ comes from the term
\[
E^{\infty}_{2np^{k+1}, 2ni-2np^{k+1}}
\]
and consequently $e_i=u_{i-p^{k+1}}\cdot e_{p^{k+1}}$, so
\begin{align*}
	p^{r+k+1}e_i&=u_{i-p^{k+1}}\cdot
	\left(p^{r+k+1}e_{p^{k+1}}\right)\\
	&= u_{i-p^{k+1}}\cdot u_{p^{k+1}}\\
	&=u_i
\end{align*}
so $(\delta_{k+1})_{\ast}$ has degree $p^{k+1}$ in
$H_{2ni}$ for $i\geqslant p^{k+1}$. This completes the proof of~2.2.
\end{proof}

Theorem~A follows by taking limits.
The cofibration $C\xrightarrow{c} G\rightarrow D$ is the limit
of $C_k\rightarrow G_k\rightarrow D_k$ and the fibrations
\begin{align*}
	T&\rightarrow W\rightarrow D\\
	F&\rightarrow D\rightarrow S^{2n+1}
\end{align*}
are the respective limits of
\begin{align*}
	T&\rightarrow W_k\rightarrow D_k\\
	\text{and}\qquad	F_k&\rightarrow D_k\rightarrow S^{2n+1}.
\end{align*}
	\section{}\label{sec3}
	In this section we will prove Theorem~B. To
	this end, consider the limiting diagram of
	the fibrations in~\textup{2.2(d)} over $k$:
\[
\begin{CD}
	T@= T\\
	@V{i}VV @VVV\\
	R@>>>W\\
	@VVV @VVV\\
	G@>>> D
\end{CD}
\]
Since $D$ is the mapping cone of the map $c\colon C\rightarrow D$,
the induced fibration over $C$ is trivial and
we have a map:
\[
C\times T\rightarrow R.
\]
Since the inclusion of $T$ into $R$ is null homotopic,
this extends to a map
\[
C\rtimes T\rightarrow R.
\]
According to \cite[Lemma A6]{AG} we get a
cofibration sequence:
\[
C\rtimes T\rightarrow R\rightarrow W
\]
\begin{theorem}\label{theor3.1}
	If $n>1$ there is a split short
	exact sequence
	\[
	0\rightarrow \widetilde{H}_{\ast}(C\rtimes T)\rightarrow
	\widetilde{H}_{\ast}(R)\rightarrow
	\widetilde{H}_{\ast}(W)\rightarrow 0.
	\]
\end{theorem}
	\begin{proof}
		We first check that the connecting
		homorphism:
		\[
		H_j(W)\rightarrow H_{j-1}(C\rtimes T)
		\]
		is trivial. By \textup{2.2(f)}, $H_j(W)\neq0$ only when
		$j=2ni-1$. However,
		\[
		C\rtimes T\simeq
		\bigvee^{\infty}_{i=1}P^{2np^i+1}\left(p^{r+i-1}\right)\rtimes
		T
		\]
		so $H_{j-1}(C\rtimes T)$ is only nonzero when $j-1=2n\ell$
		or
		$2n\ell-1$. So if $n>1$, one or the other of these groups
		is trivial. 
		To see that the sequence splits,\
		note that by \textup{2.2(f)}
		\[
		H_{2nk-1}(W)=
		\begin{cases}
			\Z_{(p)}\big/kp^{r-1}&\mbox{if}\quad k=p^s\\
			Z_{(p)}\big/kp^r&\mbox{otherwise}.
		\end{cases}
		\]
		It suffices to show that
		\begin{align*}
			kp^r\widetilde{H}_{2nk-1}(R)&=0\\
			kp^{r-1}H_{2nk-1}(R)&=0\quad\mbox{if}\ k=p^s
		\end{align*}
		Now according to \cite[Theorem C]{G4}, $R$ is a retract of
		$\Sigma T\wedge T$,
		so it suffices to prove
	\end{proof}
\begin{lemma}\label{lem3.2}
$kp^r\widetilde{H}_{2nk-1}(\Sigma T\wedge T)=0$ and if $k=p^s$
\[
kp^{r-1}\widetilde{H}_{2nk-1}(\Sigma T\wedge T)=0.
\]
\end{lemma}
\begin{proof}
	$\Sigma T\wedge T$ is a wedge of Moore spaces by \cite[4.3(m)]{GT}
	and the dimensions and orders can be read off from
	the homology groups. Let $\nu(x)$ be the number of
	powers of~$p$ in~$x$. Then
	\[
	\Sigma T\wedge T=\Sigma \bigvee_{\substack{j\geqslant 1\\
	i\geqslant 1}} P^{2ni}\left(p^{r+\nu(i)}\right)\wedge
	P^{2nj}\left(p^{r+\nu(j)}\right).
	\]
	This is a wedge of Moore spaces of dimension $2nk$
	and $2nk+1$. An element of $H_{2nk-1}(\Sigma T\wedge T)$ of order
	$p^m$ must lie in
	\[
	\Sigma P^{2ni}\left(p^{r+\nu(i)}\right)\wedge
	P^{2nj}\left(p^{r+\nu(j)}\right)
	\]
	where $k=i+j$, $m\leqslant r+\nu(i)$ and $m\leqslant r+\nu(j)$.
	Consequently $p^{m-r}$ must divide both $i$ and~$j$.
	Since $k=i+j$, $p^{m-r}$ divides $k$ and the element
	has order dividing $kp^r$. But if $k=p^s$ we
	must have $\nu(i)<s$ and $\nu(j)<s$ so $p^{m-r}$ divides
	$p^{s-1}$; i.e., $p^m$ divides $kp^{r-1}$.

	Since $R$ is a wedge of Moore spaces and the
	homomorphism
	\[
	\widetilde{H}_{\ast}(R)\rightarrow \widetilde{H}_{\ast}(W)
	\]
	has a right inverse, we can use \cite[Lemma A3]{AG}
	to construct a right homotopy inverse for the map $R\rightarrow W$.
	We obtain
\end{proof}
\begin{proposition}\label{prop3.3}
	If $n>1$, $R\simeq (C\rtimes T)\vee W$.
\end{proposition}
\begin{proposition}\label{prop3.4}
	Suppose $n>1$. Then there is a homotopy
	fibration sequence:
	\[
	C\rtimes \Omega D \rightarrow G\rightarrow D
	\]
	and $\Omega G\simeq \Omega D\times \Omega(C\rtimes \Omega D)$.
\end{proposition}
\begin{proof}
	Since the map $R\rightarrow W$ has a right homotopy
	inverse, so does
	\[
	\Omega G\simeq T\times \Omega R\rightarrow T\times \Omega W\simeq
	\Omega D.
	\]
	We use this together with \cite[Corollary~A7]{AG} to
	prove that the fiber of the map $G\rightarrow D$ is $C\rtimes
	\Omega D$.
\end{proof}
\begin{proposition}\label{prop3.5}
	The composition
	\[
	C\rtimes T\rightarrow R\rightarrow G
	\]
	factors as
	\[
	C\rtimes T\xrightarrow{1\rtimes g}C\rtimes \Omega
	G\xrightarrow{\omega}C\vee G\xrightarrow{c\vee 1}G
	\]
	where $\omega$ is the ``Whitehead product map''
	which isthe fiber of the projection $C\vee G\xrightarrow{\pi_2}G$.
\end{proposition}
\begin{proof}
	$C\rtimes T\simeq C\times T\cup\ast\times C^{\ast}(T)$ where
	$C^{\ast}(T)$
	is the reduced cone on $T$. The
	composition:
	\[
	C\times T\cup \ast\times C^{\ast}(T)\simeq C\rtimes T\rightarrow
	R\rightarrow G
	\]
is given by:
\begin{align*}
	C\times T&\xrightarrow{\pi_1} C\xrightarrow{c}G\\
	\ast\times C^{\ast}T&\xrightarrow{} \Sigma
	T\xrightarrow{\widetilde{g}}G
\end{align*}
On the other hand, the inclusion of the fiber of
the projection $C\vee G\xrightarrow{\pi_2}\nobreak G$ is given by
\begin{gather*}
	C\times \Omega G\cup\ast\times PG\simeq C\rtimes\Omega G\rightarrow
	C\vee G\\
	C\times \Omega G\xrightarrow{\pi_1}C\rightarrow C\vee G\\
	\ast\times PG\xrightarrow{e\nu}G\rightarrow C\vee G
\end{gather*}
where $PG=\{\omega\colon I\rightarrow G\mid\omega(0)=\ast\}$ and
$e\nu(\omega)=\omega(1)$.

Theorem~B follows from 3.3, 3.4, and 3.5.

It would be desirable to have a better
understanding of the restriction:
\[
W\rightarrow R\xrightarrow{\rho}G
\]
What is clear is that the composition
\[
P^{2np^k}\left(p^{r+k-1}\right)\xrightarrow{a_k}W\rightarrow D
\]
is nonzero in $\bmod\  p$ homology, so $a_k$ induces
an isomorphism in integral cohomology in
dimension $snp^k$. It seems difficult to
identify the map
\[
P^{4np^k}\left(p^{r+k}\right)\rightarrow W
\]
as it is the first class of that order.
\end{proof}
\providecommand{\bysame}{\leavevmode\hbox to3em{\hrulefill}\thinspace}
\providecommand{\MR}{\relax\ifhmode\unskip\space\fi MR }
\providecommand{\MRhref}[2]{%
  \href{http://www.ams.org/mathscinet-getitem?mr=#1}{#2}
}
\providecommand{\href}[2]{#2}

\end{document}